\documentclass{amsart}

\usepackage{lcad}
\usepackage{amsfonts}
\usepackage{amssymb}
\usepackage{enumerate}
\usepackage{amsmath}
\usepackage{graphicx}
\usepackage{amsthm}

\numberwithin{equation}{section}
\newtheorem{theorem}{Theorem}[section]
\newtheorem{lemma}[theorem]{Lemma}

\newtheorem{question}[theorem]{Question}

\newtheorem{corollary}[theorem]{Corollary}

\newtheorem{fact}[theorem]{Fact}
\theoremstyle{definition}
\newtheorem{example}[theorem]{Example}
\newtheorem{remark}[theorem]{Remark}
\newtheorem{definition}[theorem]{Definition}

\DeclareMathOperator{\supp}{supp}
\DeclareMathOperator{\pr}{pr}
\DeclareMathOperator{\diam}{diam}
\DeclareMathOperator{\inter}{int}
\DeclareMathOperator{\graph}{graph}
\DeclareMathOperator{\dist}{dist}
\DeclareMathOperator{\Lip}{Lip}

\begin{document}

\title[]{Topological Hausdorff dimension and level sets of
generic continuous functions on fractals}

\author{Rich\'ard Balka}

\address{Alfr\'ed R\'enyi Institute of Mathematics, PO Box 127, 1364
Budapest, Hungary and Eszterh\'azy K\'aroly College, Institute of
Mathematics and Informatics, Le\'anyka u. 4., 3300 Eger, Hungary}

\email{balka.richard@renyi.mta.hu}

\thanks{Supported by the
Hungarian Scientific Research Fund grant no.~72655.}

\author{Zolt\'an Buczolich}

\address{E\"otv\"os Lor\'and University,
Institute of Mathematics, P\'azm\'any P\'eter s. 1/c, 1117 Budapest,
Hungary}

\email{buczo@cs.elte.hu}


\thanks{Research supported by the Hungarian
Scientific Research Fund grant no. K075242.}

\author{M\'arton Elekes}

\address{Alfr\'ed R\'enyi Institute of Mathematics,
PO Box 127, 1364 Budapest, Hungary and E\"otv\"os Lor\'and
University, Institute of Mathematics, P\'azm\'any P\'eter s. 1/c,
1117 Budapest, Hungary}

\email{elekes.marton@renyi.mta.hu}


\thanks{Supported by the Hungarian Scientific
Foundation grants no.~72655, 83726 and J\'anos Bolyai
Fellowship.}



\begin{abstract}
In an earlier paper we introduced a new concept of dimension for
metric spaces, the so called topological Hausdorff dimension. For a
compact metric space $K$ let $\dim_{H}K$ and $\dim_{tH} K$ denote
its Hausdorff and topological Hausdorff dimension, respectively.
We proved that this new dimension describes the Hausdorff dimension of
the level sets of the generic continuous function on $K$, namely
$\sup\{ \dim_{H}f^{-1}(y) : y \in \mathbb{R} \} = \dim_{tH} K - 1$ for the
generic $f \in C(K)$, provided that $K$ is not totally disconnected, otherwise
every non-empty level set is a singleton.
We also proved that if $K$ is not totally disconnected and sufficiently
homogeneous then $\dim_{H}f^{-1}(y) = \dim_{tH} K - 1$ for the
generic $f \in C(K)$ and the generic $y \in f(K)$. The most
important goal of this paper is to make these theorems more precise.

As for the first result, we prove that the supremum is actually attained
on the left hand side of the first equation above, and also show that
there may only be a unique level set of maximal Hausdorff dimension.

As for the second result, we characterize those
compact metric spaces for which for the generic $f\in C(K)$ and the
generic $y\in f(K)$ we have $\dim_{H} f^{-1}(y)=\dim_{tH}K-1$. We also
generalize a result of B. Kirchheim by showing that if $K$ is
self-similar then for the generic $f\in C(K)$ for every $y\in \inter
f(K)$ we have $\dim_{H} f^{-1}(y)=\dim_{tH}K-1$.

Finally, we prove that the graph of the generic $f\in C(K)$ has the
same Hausdorff and topological Hausdorff dimension as $K$.
\end{abstract}

\maketitle

\section{Introduction}

 We recall first the definition of the
(small inductive) topological dimension.

\begin{definition} Set $\dim _{t} \emptyset = -1$. The \emph{topological dimension}
of a non-empty metric space $X$ is defined by induction as
$$
\dim_{t} X=\inf\{d: X \textrm{ has a basis } \mathcal{U}  \textrm{ s. t. } \dim_{t} \partial {U} \leq d-1 \textrm{ for every } U\in \mathcal{U} \}.
$$
\end{definition}

For more information on this concept see \cite{Eng} or \cite{HW}.

We introduced the topological Hausdorff dimension for compact metric
spaces in \cite{BBE}. It is  defined analogously to the topological
dimension. However, it is not inductive, and it can attain
non-integer values as well. The Hausdorff dimension of a metric
space $X$ is denoted by $\dim_{H} X$, see e.g.~\cite{F} or
\cite{Ma}. In this paper we adopt the convention that
$\dim_{H}\emptyset = -1$.

\begin{definition}\label{deftoph}
Set $\dim _{tH} \emptyset=-1$. The \emph{topological Hausdorff
dimension} of a non-empty metric space $X$ is defined as
\[
\dim_{tH} X=\inf\{d:
 X \textrm{ has a basis } \mathcal{U} \textrm{ s. t. } \dim_{H} \partial {U} \leq d-1 \textrm{ for every }
 U\in \mathcal{U} \}.
\]
\end{definition}

Both notions of dimension can attain the value $\infty$ as well.

\bigskip

Let $K$ be a compact metric space, and let $C(K)$ denote the space
of continuous real-valued functions equipped with the supremum norm.
Since this is a complete metric space, we can use Baire category
arguments. If $\dim_{t}K=0$ then the generic $f\in C(K)$ is
well-known to be one-to-one (see Lemma \ref{one-to-one}), so every non-empty
level set is a singleton.

Assume $\dim_{t}K>0$. The following results from \cite{BBE} show the
connection between the topological Hausdorff dimension and the level
sets of the generic $f\in C(K)$.

\begin{theorem} \label{ft}
If $K$ is a compact metric space with $\dim_{t}K>0$ then for the
generic $f\in C(K)$
\begin{enumerate}[(i)]
\item $\dim_{H} f^{-1} (y)\leq \dim_{tH} K-1$ for every $y\in \mathbb{R}$,
\item for every $d<\dim_{tH} K$ there exists a non-degenerate interval $I_{f,d}$
such that $\dim_{H} f^{-1} (y)\geq d- 1$ for every $y\in I_{f,d}$.
\end{enumerate}
\end{theorem}

\begin{corollary} \label{sup}
If $K$ is a compact metric space with $\dim_t K > 0$ then for the
generic $f \in C(K)$
$$\sup\{ \dim_{H}f^{-1}(y) : y \in \mathbb{R} \} = \dim_{tH} K - 1.$$
\end{corollary}

If $K$ is also sufficiently homogeneous, for example self-similar,
then we can actually say more.

\begin{theorem} \label{self-sim}
If $K$ is a self-similar compact metric space with $\dim_t K > 0$
then for the generic $f \in
C(K)$ and the generic $y \in f(K)$
$$\dim_{H}f^{-1}(y) = \dim_{tH} K - 1.$$
\end{theorem}

Theorems \ref{ft} and \ref{self-sim} are the starting points of this
paper, our primary aim is to make these theorems more precise.

\bigskip

In the Preliminaries section we introduce some notation and
definitions, cite some important properties of the topological
Hausdorff dimension, and prove several technical lemmas.

\bigskip

In Section 3 we prove a partial converse of Theorem \ref{self-sim}.
We show that for the generic $f\in C(K)$ for the generic $y\in f(K)$
we have $\dim_{H} f^{-1}(y)=\dim_{tH}K-1$ iff $K$ is homogeneous for
the topological Hausdorff dimension, that is for every non-empty
closed ball $B(x,r)\subseteq K$ we have $\dim_{tH}
B(x,r)=\dim_{tH}K$. If $K$ is (weakly) self-similar then much more
is true: For the generic $f\in C(K)$ for every $y\in \inter f(K)$ we
have $\dim_{H} f^{-1}(y)=\dim_{tH}K-1$. This generalizes a result of
B. Kirchheim. He proved in \cite{BK} that for the generic $f\in
C\left([0,1]^{d}\right)$ for every $y\in \inter f\left([0,1]^{d}
\right)$ we have $\dim_{H} f^{-1}(y)=d-1$.

\bigskip

In Section 4 we prove that the generic $f\in C(K)$ has at least one
level set of maximal Hausdorff dimension.
Hence the  supremum is attained in
Corollary \ref{sup}. We construct an attractor of an iterated
function system $K\subseteq \mathbb{R}^{2}$ such that the generic $f\in
C(K)$ has a unique level set of Hausdorff dimension $\dim_{tH}K-1$.
This shows that the above theorem is sharp.

\bigskip

Finally, in Section 5 we prove that the graph of the generic $f\in
C(K)$ has the same Hausdorff and topological Hausdorff dimension as
$K$. This generalizes a result of R. D. Mauldin and S. C. Williams
which states that the graph of the generic $f\in
C\left([0,1]\right)$ is of Hausdorff dimension one, see \cite{MW}.

\section{Preliminaries}

\subsection{Notation and definitions}

Let $(X,d)$ be a metric space, and let $A,B\subseteq X$ be arbitrary
sets. We denote by $\inter A$ and $\partial A$ the interior and
boundary of $A$. The diameter of $A$ is denoted by $\diam A$. We use
the convention $\diam \emptyset = 0$. The distance of the sets $A$
and $B$ is defined by $\dist (A,B)=\inf \{d(x,y): x\in A, \, y\in
B\}$. Let $B(x,r)=\{y\in X: d(x,y)\leq r\}$ and $U(x,r)=\{y\in X:
d(x,y)< r\}$. More generally, we define $B(A,r) = \{x \in X :
\dist(x,A) \le r \}$ and $U(A,r) = \{x \in X : \dist(x,A) < r \}$.

For two metric spaces $(X,d_{X})$ and $(Y,d_{Y})$ a function
$f\colon X\to Y$ is \emph{Lipschitz} if there exists a constant $C
\in \mathbb{R}$ such that $d_{Y}(f(x_{1}),f(x_{2}))\leq C \cdot
d_{X}(x_{1},x_{2})$ for all $x_{1},x_{2}\in X$. The smallest such
constant $C$ is called the Lipschitz constant of $f$ and denoted by
$\Lip(f)$. If $\Lip(f)<1$ then $f$ is a \emph{contraction}. A
function $f\colon X\to Y$ is called \emph{bi-Lipschitz}  if $f$ is a
bijection and both $f$ and $f^{-1}$ are Lipschitz.

If $s \ge 0$ and $\delta>0$, then
$$ \mathcal{H}^{s}_{\delta}(X)=\inf \left\{ \sum_{i=1}^\infty
(\diam U_{i})^{s}: X \subseteq \bigcup_{i=1}^{\infty} U_{i},\  \forall i\,  \diam U_i
\le \delta \right\},$$
$$ \mathcal{H}^{s}(X)=\lim_{\delta\to 0+}\mathcal{H}^{s}_{\delta}(X).$$
The \emph{Hausdorff dimension of $X$} is defined as
$$ \dim_{H} X =
\inf\left\{s \ge 0: \mathcal{H}^{s}(X) =0\right\},$$
we adopt the convention that $\dim_{H}\emptyset=-1$ throughout the
paper. For more information on these concepts see \cite{F} or
\cite{Ma}.

We define on $X\times Y$ the following metric. For all
$(x_{1},y_{1}),(x_{2},y_{2})\in X\times Y$ set
$$d_{X\times
Y}((x_{1},y_{1}),(x_{2},y_{2}))=
\sqrt{d^{2}_{X}(x_1,x_2)+d^{2}_{Y}(y_1,y_2)}.$$

The metric space $X$ is \emph{totally disconnected} if every
connected component is a singleton.

Let $X$ be a \emph{complete} metric space. A set is \emph{somewhere
dense} if it is dense in a non-empty open set, and otherwise it is
called \emph{nowhere dense}. We say that $M \subseteq X$ is
\emph{meager} if it is a countable union of nowhere dense sets, and
a set is of \emph{second category} if it is not meager. A set is
called \emph{co-meager} if its complement is meager. By Baire's
Category Theorem co-meager sets are dense. It is not difficult to show
that a set is co-meager iff it contains a dense $G_\delta$ set. We
say that the \emph{generic} element $x \in X$ has property $\mathcal{P}$ if
$\{x \in X : x \textrm{ has property } \mathcal{P} \}$ is co-meager. The
term `\emph{typical}' is also used instead of `generic'. Our main
example will be $X = C(K)$ endowed with the supremum metric (for
some compact metric space $K$).

Let $X$, $Y$ be Polish spaces. We call the set $A\subseteq X$
\emph{analytic}, if it is a continuous image of a Polish space. We
call it \emph{co-analytic} if its complement is analytic. The set
$A$ has the \emph{Baire property} if $A=U\Delta M$ where $U$ is open
and $M$ is meager. Both analytic and co-analytic sets have the Baire
property. If a set is of second category in every non-empty open set
and has the Baire property then it is co-meager. If $E\subseteq
X\times Y$, $x\in X$ and $y\in Y$ then let $E_x=\{y\in Y: (x,y)\in
E\}$ and $E^{y}=\{x\in X: (x,y)\in E\}$. Let $\pr_{X}\colon X\times
Y\to X$, $\pr_X(x,y)=x$ be the projection of $X\times Y$ onto $X$.
If $E\subseteq X\times Y$ is Borel, then $\pr_{X}(E)$ is analytic.
For more information see \cite{Ke}.

If $K$ is a non-empty compact metric space then we say that $K$ is
an \emph{attractor of an iterated function system} (IFS) if there
exist contractions $\Psi_{i}\colon K\rightarrow K$, $i\in \{1,\dots
,m\}$ such that $K=\bigcup _{i=1} ^{m} \Psi _{i}(K)$. If the
$\Psi_{i}$'s are also similarities then $K$ is \emph{self-similar}.

For every $\alpha \in (0,1)$ we construct the \emph{middle-$\alpha$
Cantor set} $C_{\alpha}$ in the following way. In the first step we
remove the middle-$\alpha$ open interval $((1-\alpha)/2,(1+\alpha)/2)$ from $[0,1]$. After the $(n-1)$st step
we have $2^{n-1}$ disjoint, closed $(n-1)$st level intervals. In the
$n$th step we remove the middle-$\alpha$ open intervals from each of
them. We continue this procedure for all $n\in \mathbb{N}^{+}$, and the
limit set is the middle-$\alpha$ Cantor set. It is well-known that
$\dim_{H} C_{\alpha}=\log 2/\log (2/(1-\alpha))$.

Let us define the \emph{Smith-Volterra-Cantor set} $S$ in the
following way. In the first step we remove the open interval of
length $ 1/4 $ from the middle of $[0,1]$. After the $(n-1)$st step
we have $2^{n-1}$ disjoint, closed $(n-1)$st level intervals. In
the $n$th step we remove the middle open intervals of length
$1/2^{2n}$ from each of them. We continue this procedure for all
$n\in \mathbb{N}^{+}$, and the limit set is the Smith-Volterra-Cantor set.
Elementary computation shows that $S$ has positive Lebesgue measure
(more precisely its measure is $1/2$).

The \emph{$n$th level elementary pieces} of $C_{\alpha}$ are the
intersections of $C_{\alpha}$ with the $n$th level intervals of
$C_{\alpha}$. This definition is also analogous for $S$.

We adopt the convention that \emph{intervals} are non-degenerate.

\subsection{Properties of the topological Hausdorff dimension}

The next theorems are from \cite{BBE}.

\begin{fact} \label{equiv} For every metric space $X$
$$\dim_{tH} X=0 \Longleftrightarrow \dim_{t} X=0.$$  \end{fact}

\begin{theorem} \label{<} For every metric space $X$
$$\dim_{t}X\leq \dim_{tH} X \leq \dim _{H} X.$$  \end{theorem}

\begin{theorem} \label{prop} The topological Hausdorff dimension satisfies
the following properties.
\begin{enumerate}[(i)]
\item \textbf{Extension of the classical dimension.} The topological Hausdorff dimension of a
countable set equals zero, and for open subspaces of $\mathbb{R}^{d}$ and
for smooth $d$-dimensional manifolds the topological Hausdorff
dimension equals $d$.
\item \textbf{Monotonicity.}  If $X\subseteq Y$ are metric
spaces then $\dim_{tH} X \leq \dim_{tH} Y$.
\item \textbf{Lipschitz-invariance.}  Let $X,Y$ be metric spaces. If $f\colon X\to Y$ is a
Lipschitz homeomorphism then $\dim_{tH}Y\leq \dim_{tH}X$. If $f$ is
bi-Lipschitz then $\dim_{tH}X=\dim_{tH}Y$.
\item  \textbf{Countable stability for closed sets.} Let $X$ be a separable metric
space and $X=\bigcup_{n\in \mathbb{N}} \ X_{n}$, where $X_{n},$ $n\in \mathbb{N}$ are
closed subsets of $X$. Then $\dim_{tH} X=\sup_{n\in \mathbb{N}}
\dim_{tH}X_{n}$.
\end{enumerate}
 \end{theorem}

\begin{theorem} \label{prod} If $X$ is a non-empty separable metric space then
$$\dim_{tH}\left(X \times [0,1]\right)=\dim_{H} X+1.$$
\end{theorem}

For compact metric spaces the infimum is attained in the definition
of the topological Hausdorff dimension.

\begin{theorem} \label{min} If $K$ is a non-empty compact metric space then
$$\dim_{tH} K=\min\{d:
 K \textrm{ has a basis } \mathcal{U}  \textrm{ s. t. } \dim_{H} \partial {U} \leq d-1 \textrm{ for every } U\in \mathcal{U}
\}.$$
 \end{theorem}

\subsection{Technical lemmas}

The following lemma is folklore, but we could not find a reference, so we
outline a short proof.

\begin{lemma} \label{one-to-one}
Let $K$ be a compact metric space with $\dim_t K=0$. Then the generic $f\in C(K)$ is one-to-one.
\end{lemma}

\begin{proof} Let $\mathcal{U}$ be a countable basis of $K$ consisting of clopen sets. For every $U,V\in \mathcal{U}$, $U\cap V=\emptyset$ consider
$$\mathcal{F}_{U,V}=\{f\in C(K): f(U)\cap f(V)=\emptyset\}.$$
By compactness of $U$ and $V$, the sets $\mathcal{F}_{U,V}$ are open. They are also dense, in fact
this is witnessed by functions of finite range. Thus the countable intersection
$$\mathcal{F}=\bigcap_{U,V\in \mathcal{U},~ U\cap V=\emptyset} \mathcal{F}_{U,V}$$
is co-meager in $C(K)$, and clearly every $f\in \mathcal{F}$ is one-to-one.
\end{proof}

The next lemma and its consequence will be very useful throughout
the paper.

\begin{lemma} \label{catlem}

Let $X,Y$ be complete metric spaces and let $R\colon X\to Y$ be a
continuous, open and surjective mapping.
\begin{enumerate}[(i)]
\item  If $A\subseteq X$ is of second category/co-meager then
$R(A)\subseteq Y$ is of second category/co-meager.
\item  If
$B\subseteq Y$ is of second category/co-meager then
$R^{-1}(B)\subseteq X$ is of second category/co-meager.
\end{enumerate}
\end{lemma}

 \begin{proof}

$(i)$ First we show that if $B\subseteq Y$ is meager then
$R^{-1}(B)\subseteq X$ is also meager. Clearly it is enough to prove
that if $B\subseteq Y$ is closed and nowhere dense then
$R^{-1}(B)\subseteq X$ is nowhere dense.
Since $R$ is continuous $R^{-1}(B)$ is closed.
 We show that $R^{-1}(B)$ is nowhere dense.
Assume to the contrary that there is a non-empty open set
$U\subseteq R^{-1}(B)$. Since the map $R$ is open the set
 $R(U)$ is
non-empty and open. Then $R(U)\subseteq B$ implies that $B$ is of
second category, a contradiction.

Let $A\subseteq X$ be of second
category. Assume to the contrary that $R(A)\subseteq Y$ is meager.
Then by the previous argument $R^{-1}(R(A))$ is meager and $A\subseteq
R^{-1}(R(A))$, a contradiction.

Suppose that  $A\subseteq X$ is co-meager. We want to prove that
$R(A)\subseteq Y$ is also co-meager. We may assume that $A$ is a
dense $G_{\delta}$ set. Assume to the contrary that $R(A)$ is not
co-meager. As a continuous image of a Borel set $R(A)$ is analytic,
and hence has the Baire property. Thus there exists a non-empty open set
$U\subseteq Y$ such that $R(A)\cap U$ is meager. Since $R$ is continuous
and surjective $R^{-1}(U)$ is open and non-empty. The map
$\widehat{R}=R|_{R^{-1}(U)}:R^{-1}(U)\to U$ is clearly continuous,
open and surjective. Since $R(A)\cap U$ is meager $\widehat{R}^{-1}(R(A)\cap U)$ is meager
in $R^{-1}(U)$.  The set $A\cap
R^{-1}(U)$ is co-meager in $R^{-1}(U)$, and clearly $A\cap
R^{-1}(U)\subseteq \widehat{R}^{-1}(R(A)\cap U)$, a contradiction.

$(ii)$ Let $B\subseteq Y$ be of second category. Assume to the
contrary that $R^{-1}(B)$ is  meager. Then $R^{-1}(B)^{c}$ is
co-meager and its $R$ image $R(R^{-1}(B)^{c})\subseteq B^{c}$ is not
co-meager. This contradicts   part (i) of the lemma.

Let $B\subseteq Y$ be co-meager. Then $B^{c}$ is meager, and hence
$R^{-1}(B^{c})$ is meager. This implies that $R^{-1}(B)=X\setminus
R^{-1}(B^{c})$ is co-meager.
\end{proof}

We need the following special case.

\begin{corollary} \label{catlem2}

Let $K_1\subseteq K_2$ be compact metric spaces and
$$R\colon C(K_2)\rightarrow
 C(K_1), ~ R(f)=f|_{K_1}.$$

\begin{enumerate}[(i)]
\item  If $\mathcal{F}_2\subseteq C(K_2)$ is of second category/co-meager then $
R(\mathcal{F}_2)\subseteq C(K_1)$ is of second category/co-meager.
\item  If
$\mathcal{F}_1\subseteq C(K_1)$ is of second category/co-meager then $
R^{-1}(\mathcal{F}_1)\subseteq C(K_2)$ is of second category/co-meager.
\end{enumerate}
\end{corollary}

\begin{proof} Clearly $C(K_2)$ and $C(K_1)$ are complete metric spaces, $R$
is continuous, and Tietze's Extension Theorem implies that $R$ is
surjective and open. Thus Lemma \ref{catlem} completes the proof.
\end{proof}

We need the following theorem, see \cite[6.1. Thm.]{MM} for the
proof.

\begin{theorem} \label{tmm} Let $X,Y$ be Polish spaces, and let $E\subseteq
X\times Y$ be a Borel set. If $E_{x}$ is $\sigma$-compact for all
$x\in X$ then the function $h\colon X\rightarrow [-1,\infty]$ defined by
$h(x)=\dim_{H} E_{x}$ is Borel measurable.  \end{theorem}

\begin{remark} One has to be a little careful with applying this theorem
in our situation, since, unlike \cite{MM}, we adopt the convention that
$\dim_{H}
\emptyset=-1$. This means that the level sets of $h$ may need to be modified
by the set $\{x\in X: E_x=\emptyset\}=(\pr_{X} (E))^{c}$. Therefore, besides
referring to the proof in \cite{MM}, we
also have to verify that $\pr_{X} (E)$ is Borel. But as $E$ is Borel and
$E_{x}$ is $\sigma$-compact for all
$x\in X$, this follows from \cite[18.18. Thm.]{Ke}.
\end{remark}

\begin{lemma} \label{borel}
Let $K$ be a compact metric space and $d\in \mathbb{R}$. Then the set
$$\Delta=\left\{(f,y)\in C(K)\times \mathbb{R}: \dim_{H} f^{-1}(y)< d\right\}$$
is Borel.
\end{lemma}

\begin{proof}
We check that the conditions of Theorem \ref{tmm} hold for $X=
C(K)\times \mathbb{R}$, $Y=K$ and $E= \{(f,y,z)\in C(K)\times \mathbb{R}\times K:
f(z)=y\}\subseteq X\times Y$. Clearly $X,Y$ are Polish spaces and
$E$ is closed, thus Borel. For every $(f,y)\in X$ the set
$E_{(f,y)}=\{z\in K: f(z)=y\}=f^{-1}(y)$ is compact. Therefore
Theorem \ref{tmm} implies that $h\colon X\rightarrow [0,\infty]$,
$h((f,y))=\dim_{H}E_{(f,y)}=\dim_{H}f^{-1}(y)$ is Borel measurable.
Thus $h^{-1}\left((-\infty, d)\right)=\left\{(f,y)\in C(K)\times
\mathbb{R}: \dim_{H} f^{-1}(y)<d \right\}=\Delta$ is Borel.
\end{proof}

\begin{lemma} \label{cpt} Suppose $(K,d)$ is a compact metric space such that for
all $x\in K$ and $r>0$ we have $\dim_{t}B(x,r)>0$. Let $\mathcal{C}$ be the
set of connected components of $K$. Then for the generic $f\in C(K)$
$$\bigcup_{C\in \mathcal{C}} \inter f(C)=\inter f(K).$$
\end{lemma}

\begin{remark} If $K_{0}$ is the triadic Cantor set then $K=K_{0}\times [0,1]$
has uncountably many connected components but it is a `homogeneous'
self-similar set. Lemma \ref{cpt} actually holds for arbitrary compact metric
spaces, but we will not need this fact.
\end{remark}

\begin{proof}[Proof of Lemma \ref{cpt}]
Consider $$\mathcal{F}=\left\{f\in C(K): \bigcup_{C\in \mathcal{C}} \inter
f(C)=\inter f(K)\right\},$$
and for all $n\in \mathbb{N}^{+}$ let
$$\mathcal{F}_{n}=\left\{f\in C(K): \forall y\in f(K)\setminus B\left(\partial f(K),1/n\right), ~
\exists C\in \mathcal{C} \textrm{ s. t. } y\in \inter f(C)\right\}.$$
We must prove that $\mathcal{F}$ is co-meager in $C(K)$. Since
$\mathcal{F}=\bigcap_{n\in \mathbb{N}^{+}} \mathcal{F}_{n}$, it is enough to show that the
$\mathcal{F}_{n}$'s are co-meager in $C(K)$. Let us fix $n\in  \mathbb{N}^{+}$ and
let $f_0\in C(K)$ and $\varepsilon>0$ be arbitrary. It is sufficient
to show that there is a non-empty ball $B(g_{0},r_0)\subseteq
\mathcal{F}_{n}\cap B(f_0,4\varepsilon)$.

Since  $f_{0}$ is uniformly continuous on $K$ there is a
$\delta_1>0$ such that if $x,z\in K$ and $d(x,z)\leq \delta_1$ then
$|f_{0}(x)-f_{0}(z)|\leq \varepsilon$. By the compactness of $K$
there is a finite set $\{ x_{1},...,x_{k} \}$ such that
$\bigcup_{i=1}^{k}B(x_{i},\delta_1)=K$. Choose $0<\delta_{2}<\delta_1$
such that the balls $B(x_{i},\delta_{2})$ are disjoint. The
conditions of the lemma imply that for every $i\in \{1,\dots,k\}$ we
have $\dim_{t} B(x_{i},\delta_{2}/2)>0$. Thus there exist
non-trivial connected components $C_i$ of $B(x_{i},\delta_{2}/2)$
for all $i\in \{1,\dots,k\}$, see \cite[6.2.9. Thm.]{RE}. For all
$i\in \{1,\dots, k\}$ let us choose $u_{i},v_{i}\in C_{i}$, $u_i\neq
v_i$ and select $\varepsilon_{i}\in [\varepsilon, 2\varepsilon]$
such that the set
$$E=\left\{f_{0}(x_i)+ \varepsilon_{i}: i=1, \dots,
k \right\}\cup \left\{f_{0}(x_i)- \varepsilon_{i}: i=1, \dots, k
\right\}$$
has $2k$ many elements. Let $\theta=\min\{d(x,y): x,y\in E, x\neq
y\}>0$. Clearly for all $x\in B(x_i,\delta_1)$, $i\in \{1,\dots,k\}$
we have $f_{0}(x)\in [f_{0}(x_i)-\varepsilon,
f_{0}(x_i)+\varepsilon]\subseteq [f_{0}(x_i)-\varepsilon_i,
f_{0}(x_i)+\varepsilon_i]$. Hence Tietze's Extension Theorem implies
that there exists a $g_0\in C(K)$ such that $g_0(x)=f_0(x)$ if $x\in
K\setminus \bigcup_{i=1}^{k} U(x_i,\delta_2)$ and for all $i\in
\{1,\dots, k\}$ we have $g_0(u_i)=f_0(x_i)-\varepsilon_i$,
$g_0(v_i)=f_0(x_i)+\varepsilon_i$ and
\begin{equation} \label{eq1}
g_{0}(x)\in \left[f_0(x_i)-\varepsilon_i,f_0(x_i)+\varepsilon_i
\right], \quad x\in B(x_i,\delta_1).
\end{equation}
Therefore, using that the oscillations of $f_0$ on the
$B(x_i,\delta_1)$'s are at most $\varepsilon$ and $\varepsilon_i\leq
2\varepsilon$ for all $i\in \{1,\dots,k\}$, we have $g_0\in
B(f_0,3\varepsilon)$. Set $r_0=\min \left\{\varepsilon, \theta/6,
1/(3n) \right\}$. Since $B(g_0,r_0)\subseteq
B(g_0,\varepsilon)\subseteq B(f_0,4\varepsilon)$,  it is enough to
prove that $B(g_0,r_0)\subseteq \mathcal{F}_{n}$. Let $f\in B(g_{0},r_0)$
and $y_0\in f(K)\setminus B(\partial f(K),1/n)$, that is,
$B(y_0,1/n)\subseteq \inter f(K)$. It is enough to verify that there
is an $i\in \{1,\dots, k\}$ such that $y_0\in \inter f(C_i)$. (Note
that every $C_i$ is contained in a member of $\mathcal{C}$.) Let us choose
$z_0 \in K$ with $f(z_0)=y_0$ and fix $i \in \{1, \dots ,k\}$ such
that $z_0\in B(x_{i}, \delta_1)$. Then equation $\eqref{eq1}$ and
$f\in B(g_0,r_0)$ imply that $y_0\in [f_{0}(x_i)-\varepsilon_i -r_0,
f_{0}(x_i)+\varepsilon_i+r_0]$.

First assume that $y_0\in (f_{0}(x_i)-\varepsilon_i +r_0,
f_{0}(x_i)+\varepsilon_i-r_0)=(g_{0}(u_i)+r_0,g_0(v_i)-r_0)$. Then
$f\in B(g_0,r_0)$ and the connectedness of $C_i$ imply $y_0\in
(f(u_i),f(v_i))\subseteq \inter f(C_i)$.

Finally, suppose that $y_0\in [f_{0}(x_i)-\varepsilon_i -r_0,
f_{0}(x_i)-\varepsilon_i+r_0]$ or
\begin{equation} \label{eq2} y_0\in [f_{0}(x_i)+\varepsilon_i -r_0,
f_{0}(x_i)+\varepsilon_i+r_0].
\end{equation}
We may assume by symmetry that \eqref{eq2} holds. Since $y_0+3r_0\in
B(y_0, 1/n)\subseteq \inter f(K)$, there exists $z_1\in K$ such that
$f(z_1)=y_0+3r_0$ and $j\in \{1,\dots,k\}$ such that $z_1\in
B(x_j,\delta_1)$. From $f\in B(g_0,r_0)$ and \eqref{eq1} it follows
that
\begin{equation} \label{eq3} y_0+3r_0\in \left[f_0(x_j)-\varepsilon_j-r_0,
f_0(x_j)+\varepsilon_j+r_0\right].
\end{equation}
Equation \eqref{eq2} implies
$y_0+3r_0>f_{0}(x_i)+\varepsilon_i+r_0$, thus we have $j\neq i$.
Equation \eqref{eq2} also implies
$|y_0-(f_{0}(x_i)+\varepsilon_i)|\leq r_0$. Therefore the triangle
inequality and the definition of $\theta$ yield
\begin{align}  \label{eq4} \left|y_0-(f_0(x_j)- \varepsilon_j)\right|&\geq
\left|(f_0(x_j)- \varepsilon_j)-(f_0(x_i)+\varepsilon_i)\right|-
\left|y_0-(f_{0}(x_i)+\varepsilon_i)\right| \notag \\
&\geq \theta-r_0>4r_0.
\end{align}
Then \eqref{eq3} implies $y_0< f_0(x_j)+\varepsilon_j-r_0$ and
$y_0\geq f_0(x_j)-\varepsilon_j-4r_0$, thus \eqref{eq4} yields
$y_0\in
\left(f_0(x_j)-\varepsilon_j+r_0,f_0(x_j)+\varepsilon_j-r_0\right)=\left(g_{0}(u_j)+r_0,g_0(v_j)-r_0\right).$
Hence $f\in B(g_0,r_0)$ and the connectedness of $C_j$ imply $y_0\in
(f(u_j),f(v_j))\subseteq \inter f(C_j)$. This completes the proof.
\end{proof}

\begin{lemma} \label{lmax1} Let $K$ be a compact metric space with a fixed $x_0\in
K$. Let $K_n\subseteq K,$ $n\in \mathbb{N}$ be compact sets such that
\begin{enumerate}[(i)]

\item  $\dim_{t}K_n>0$ for all $n\in \mathbb{N}$ and

\item $\diam \left(K_n\cup \{x_0\}\right)\to 0$ if $n\to
\infty$.
\end{enumerate}

Then for the generic $f\in C(K)$ we have $f(x_0) \in f(K_n)$ for
infinitely many $n\in \mathbb{N}$.
\end{lemma}

\begin{proof} Clearly it is enough to
show that the sets $$\mathcal{F}_{N}=\{f\in C(K): f(x_0) \notin f(K_n) \textrm{
for all } n\geq N \}$$ are nowhere dense in $C(K)$ for all $N\in
\mathbb{N}$. Let $f_0\in C(K)$ and $\varepsilon>0$ be arbitrary, it is
enough to find a ball in $B(f_0,2\varepsilon)\setminus \mathcal{F}_{N}$. The
compact $K_n$'s have positive topological dimension, therefore they are not
totally disconnected, see \cite[6.2.9. Thm.]{RE}. Let us choose a
non-trivial connected component $C_n\subseteq K_n$ for every $n\in
\mathbb{N}$.  We can choose by $(ii)$ an $n_0\in \mathbb{N}$ such that $n_0\geq
N$ and $\diam f_0\left(C_{n_0}\cup \{x_0\}\right)<\varepsilon$.
Tietze's Extension Theorem implies that there is an $f \in
B(f_0,\varepsilon)$ such that $\diam f(C_{n_0})>0$ and $f(x_0)$ is
the midpoint of $f(C_{n_0})$. If $\delta=\min\left\{\varepsilon,
\frac {1}{4}\diam f(C_{n_0})\right\}$ then for all $g\in
B(f,\delta)$ we have $g(x_0)\in g(C_{n_0})\subseteq g(K_{n_0})$, so
$g\notin \mathcal{F}_{N}$. Thus $B(f,\delta) \subseteq
B(f_0,2\varepsilon)\setminus \mathcal{F}_{N}$.
\end{proof}

The following lemma is probably known, but we could not find an
explicit reference, so we outline its proof.

\begin{lemma} The Smith-Volterra-Cantor set $S$ is an attractor of an IFS.
\end{lemma}

\begin{proof} In the $n$th step of the construction we remove
$2^{n-1}$ many disjoint open intervals of length $a_{n}=1/2^{2n}$,
the remaining $2^{n}$ disjoint, closed $n$th level intervals are of
length $b_{n}=\frac{1}{2^n}\left(1-\sum
_{i=1}^{n}2^{i-1}a_{i}\right)=1/2^{n+1}+1/2^{2n+1}$. Let $\pi\colon
S\to \{0,1\}^{\mathbb{N}}$ be the natural homeomorphism, that is,
for $x\in S$ and $n\in \mathbb{N}$ we define $\pi(x)(n)$, as
follows. There is a unique $n$th level interval $I_n$ and a unique
$(n+1)$st level interval $I_{n+1}$ such that $x\in I_n$ and $x\in
I_{n+1}$. Then $I_{n+1}$ is either the left or the right hand side
interval of $I_{n}$. If it is the left hand side interval then
$\pi(x)(n)=0$, otherwise $\pi(x)(n)=1$. Let
\begin{align}
 &\varphi _{1}\colon S\to S\cap \left[0, 1/2 \right], ~
\varphi _{1}(x) = \pi^{-1}\left(0 \ \! \hat{} \ \! \pi(x)\right), \notag \\
&\varphi_{2}\colon S\to S\cap \left[ 1/2 ,1 \right], ~ \varphi
_{2}(x) = \pi^{-1}\left(1 \ \! \hat{} \ \! \pi(x)\right)  \label{fidef}
\end{align}
be the natural homeomorphisms onto the left and the right half of
$S$ (where $\hat{}$ stands for concatenation). Clearly, $S=\varphi
_{1}(S)\cup \varphi _{2}(S)$, so it is sufficient to prove that
$\varphi _{1}$ and $\varphi_{2}$ are contractions. By symmetry it is
enough to show that $\varphi _{1}$ is a Lipschitz map with
$\Lip(\varphi _{1})\leq  1/2$, that is, for all $x,z\in S$
\begin{equation} \label{fi} |\varphi _{1}(x)-\varphi _{1}(z)|\leq
\frac{|x-z|}{2}.  \end{equation}
The endpoints of the intervals at the construction are dense in $S$.
Thus we may assume for the proof of \eqref{fi} that $x,z$ are both
endpoints of some $n$th level intervals and $x<z$. Let us assume
that in the interval $[x,z]$ there are $\beta_n=\beta_{n,x,z}$ many
intervals of length $b_{n}$ and there are
$\alpha_{i}=\alpha_{i,x,z}$ many open intervals of length $a_{i}$,
$i\in \{1,\dots,n\}$.  In the interval $[\varphi _{1}(x),\varphi
_{1}(z)]$ there are  $\beta_n$ many closed intervals of length
$b_{n+1}$ and there are $\alpha_{i}$ many open intervals of length
$a_{i+1}$, $i\in \{1,\dots ,n\}$. These intervals are disjoint, and
their union is $[x,z]$ and $[\varphi _{1}(x),\varphi _{1}(z)]$
(apart from the endpoints $x,z$ and
$\varphi_{1}(x),\varphi_{1}(z)$), respectively. We obtain
$|x-z|=\beta _{n} b_{n}+\sum _{i=1}^{n} \alpha _{i} a_{i}$ and
$|\varphi _{1}(x)-\varphi _{1}(z)|=\beta _{n} b_{n+1}+\sum
_{i=1}^{n} \alpha _{i} a_{i+1}$. Hence for \eqref{fi} it is enough
to prove that $b_{n+1}\leq b_{n}/2$ and $a_{i+1}\leq a_{i}/2$ for
all $i\in \{1,\dots,n\}$, but it is clear from the definitions of
the $b_{n}$'s and the $a_{n}$'s.
\end{proof}

\section{Level sets on fractals}

Let $K$ be a compact metric space. If $\dim_{t}K=0$ then the generic continuous function is one-to-one on
$K$ by Lemma \ref{one-to-one}, hence every non-empty level set is a single point.

Thus in the sequel we  assume that $\dim_{t}K>0$.

\begin{definition} \label{suppd}
If $K$ is a compact metric space then let
$$\supp  K=\left\{x\in K: \forall r>0,
\  \dim_{tH} B(x,r)=\dim_{tH} K\right\}.$$
We say that $K$ is \emph{homogeneous for the topological Hausdorff
dimension} if $\supp K=K$.
\end{definition}

\begin{remark} \label{rsupp}
The stability of the topological Hausdorff dimension for closed sets
clearly yields $\supp K \neq \emptyset$. If $K$ is self-similar then
it is also homogeneous for the topological Hausdorff dimension.
\end{remark}

 We proved in \cite{BBE} that if
$K$ is homogeneous for the topological Hausdorff dimension then for
the generic $f\in C(K)$ for the generic $y\in f(K)$ we have
$\dim_{H} f^{-1}(y)=\dim_{tH}K-1$. Now we prove the opposite
direction.

\begin{theorem} \label{tip} Let $K$ be a compact metric space with $\dim_{t}K>0$. The following
statements are equivalent.
\begin{enumerate}[(i)]
\item  For the generic $f\in C(K)$ for the generic $y\in f(K)$ we have $\dim_{H}
f^{-1}(y)=\dim_{tH}K-1$.
\item  $K$ is homogeneous for the topological Hausdorff dimension.
\end{enumerate}
\end{theorem}

\begin{proof} $(ii)\Rightarrow (i)$: See \cite[Thm. 6.22.]{BBE}.

$(i)\Rightarrow (ii)$: Assume to the contrary that $K\setminus \supp
K\neq \emptyset$. Then there exist $f_0\in C(K)$ and
$\varepsilon_0>0$ such that for all $f\in B(f_{0}, \varepsilon_0)$
we have $f(K)\setminus f\left(\supp K\right)\neq \emptyset$. Let us
choose for all $f\in  B(f_{0}, \varepsilon_0)$ an interval $I_f$
such that $ I_f \cap f\left(\supp K\right) =\emptyset$ and $I_f\cap
f(K\setminus \supp K) \neq \emptyset$. Let us define for all $n\in
\mathbb{N}^{+}$
$$K_n=\{x\in K: \dist(x,\supp K)\geq 1/n\}.$$
Then the $K_n$'s are compact and $\bigcup_{n\in \mathbb{N}^{+}} K_n=K\setminus
\supp K$. The definition of $\supp K$ and the compactness of $K_n$
imply that $K_n$ can be covered with finitely many closed balls of
topological Hausdorff dimension less than $\dim_{tH}K$. Then the
stability of the topological Hausdorff dimension for closed sets
implies
\begin{equation} \label{extra} \dim_{tH}K_{n}<\dim_{tH}K \quad (n\in \mathbb{N}^{+}).
\end{equation}
For all $n\in \mathbb{N}^{+}$ let
$$\mathcal{F}_{n}=\left\{f\in C(K_n): \dim_{H}f^{-1}(y)\leq \dim_{tH}K_n-1 \textrm{ for all } y\in
\mathbb{R}\right\}.$$
Define $R_{n}\colon K\to K_n$, $R_n(f)=f|_{K_n}$ and let
$\mathcal{F}=\bigcap_{n\in \mathbb{N}^{+}}R_{n}^{-1}(\mathcal{F}_{n})$. Theorem \ref{ft}
yields that the $\mathcal{F}_n$'s are co-meager in $C(K_n)$, and it follows
from Corollary \ref{catlem2}  that the $R_{n}^{-1}(\mathcal{F}_{n})$'s are
co-meager in $C(K)$. As $\mathcal{F}$ is the intersection of countable many
co-meager sets, it is also co-meager in $C(K)$. If $f\in
B(f_0,\varepsilon)$ and $y\in I_f\cap f(K)$ then the definition of
$I_f$ and the compactness of $f^{-1}(y)$ imply that there is an
$n_{f,y}\in \mathbb{N}^{+}$ such that $f^{-1}(y)\subseteq K_{n_{f,y}}$. If
$f\in \mathcal{F}$ then for all $y\in I_f\cap f(K)$ the definition of
$n_{f,y}$, the definition of $\mathcal{F}$ and \eqref{extra} imply
\begin{align*} \dim_{H} f^{-1}(y)&=\dim_{H} \left(f^{-1}(y)\cap
K_{n_{f,y}}\right) \leq \dim_{tH}K_{n_{f,y}}-1 \\
&<\dim_{tH}K-1.
\end{align*}
This contradicts $(i)$, and the proof is complete.
\end{proof}

B. Kirchheim showed in \cite{BK} that for the generic $f\in
C\left([0,1]^{d}\right)$ for every $y\in \inter f\left([0,1]^{d}
\right)$ we have $\dim_{H} f^{-1}(y)=d-1$. We generalize this result
for weakly self-similar compact metric spaces.

\begin{definition} \label{defweakselfsim} Let $K$ be a compact metric space. We say
that $K$ is \emph{weakly self-similar} if for all $x\in K$ and $r>0$
there exist a compact set $K_{x,r}\subseteq B(x,r)$ and a
bi-Lipschitz map $\phi_{x,r}\colon K_{x,r} \to K$.
\end{definition}

\begin{remark} If $K$ is self-similar then it is also weakly
self-similar. If $K$ is weakly self-similar then it is also
homogeneous for the topological Hausdorff dimension.
\end{remark}

\begin{theorem} \label{thinty} Let $K$ be a weakly self-similar compact metric
space. Then for the generic $f\in C(K)$ for any $y\in \inter f(K)$
$$\dim_{H} f^{-1}(y)=\dim_{tH}K-1.$$  \end{theorem}

\begin{proof} If $\dim_{t} K=0$ then the generic $f\in
C(K)$ is one-to-one, and $f(K)$ is nowhere dense. Thus $\inter
f(K)=\emptyset$, and the statement is obvious.

Next we  assume $\dim_{t}K>0$. Theorem \ref{ft} implies that for the
generic $f\in C(K)$ for all $y\in \mathbb{R}$ we have $\dim_{H}
f^{-1}(y)\leq \dim_{tH} K-1$, thus we only need to verify the
opposite inequality.

Fact \ref{equiv} implies $\dim_{tH}K>0$.
It follows from the weak self-similarity of
$K$  that for all $x\in K$ and $r>0$ we have
$\dim_{tH}B(x,r)=\dim_{tH} K>0$. Then applying Fact \ref{equiv}
again we obtain that $\dim_{t}B(x,r)>0$. If $\mathcal{C}$ denotes the set of
connected components of $K$ then Lemma \ref{cpt} yields that for the
generic $f\in C(K)$ we have $\bigcup_{C\in \mathcal{C}} \inter f(C)=\inter
f(K)$.

Thus it is enough to prove that for the generic $f\in C(K)$ for
every $y\in \bigcup_{C\in \mathcal{C}} \inter f(C)$ we have $\dim_{H}
f^{-1}(y)\geq \dim_{tH}K-1$.

Let us choose a sequence $0<d_n\nearrow \dim_{tH}K$ and let us fix
$n\in \mathbb{N}^{+}$. Theorem \ref{ft} implies that for the generic $f\in
C(K)$ there exists an interval $I_{f,d_n}$ such
that for all $y\in I_{f,d_n}$ we have $\dim_{H} f^{-1}(y)\geq d_n-1$.
By Baire's Category Theorem there are $m_2<m_{1}<M_{1}<M_2$ such
that
$$ \mathcal{H}_n=  \{ f\in C(K): f(K)\subseteq [m_2,M_2],~ \forall
y\in [m_{1},M_{1}],\  \dim_{H} f^{-1}(y)\geq d_n-1 \}$$
is of second category. Note that $d_n>0$ implies that for every
$f\in \mathcal{H}_n$ we have $[m_1,M_1]\subseteq f(K)$. Let us also define
 the following set.
$$ \mathcal{G}_{n} = \left\{f\in C(K): \forall y\in
\bigcup_{C\in \mathcal{C}} \left(f(C)\setminus B(\partial f(C),1/n)\right),
~ \dim_{H} f^{-1}(y)\geq d_n-1\right\}.
$$
It is sufficient to verify that $\mathcal{G}_n$ is co-meager, since by
taking the intersection of the sets $\mathcal{G}_n$ for all $n\in \mathbb{N}^{+}$
we obtain the desired co-meager set in $C(K)$. In order to prove
this we show that $\mathcal{G}_n$ contains `certain copies' of $\mathcal{H}_n$.
First we need the following lemma.

\begin{lemma} \label{baire} $\mathcal{H}_{n}$ and $\mathcal{G}_n$ have the Baire property. \end{lemma}

\begin{proof}[Proof of Lemma \ref{baire}] Lemma \ref{borel}
implies that $\Gamma_n= \{(f,y)\in C(K)\times \mathbb{R}: \dim_{H}
f^{-1}(y)< d_n-1\}$ is Borel. Then $\mathcal{H}_{n}=\{f\in C(K):
f(K)\subseteq [m_{2},M_{2}]\}\cap \{f\in C(K): \forall y\in
[m_{1},M_{1}],\, \dim_{H} f^{-1}(y)\geq d_n-1\}$. The first term of
the intersection is clearly closed. It is sufficient to prove that
the second one has the Baire property. It equals
$\left(\pr_{C(K)}\Big (\left(C(K)\times [m_{1},M_{1}]\right)\cap
\Gamma_n \Big )\right)^{c}$, which is the complement of the
projection of a Borel set. Hence it is co-analytic, and therefore
has the Baire property.

The set
$$\Delta_n= \left\{ (f,y)\in C(K)\times \mathbb{R}: y\in
\bigcup_{C\in \mathcal{C}} \Big (f(C)\setminus B(\partial f(C),1/n)\Big )
\right\}$$
is clearly open. Then $\mathcal{G}_n=\left(\pr_{C(K)}\left(\Gamma_n \cap
\Delta_n\right)\right)^{c}$, which is the complement of the
projection of a Borel set. Thus it is co-analytic, and therefore has
the Baire property.
\end{proof}

Now we return to the proof of Theorem \ref{thinty}. Consider
$\mathcal{G}_n$ (note that we already fixed $n$), our aim is to
show that $\mathcal{G}_n$ is co-meager. Since $\mathcal{G}_n$ has the Baire
property, it is enough to prove that $\mathcal{G}_n$ is of second category
in every non-empty open subset of $C(K)$. Let $f_0\in C(K)$ and
$0<\varepsilon<1/n$ be fixed. We want to show that $\mathcal{G}_n \cap
B(f_{0},\varepsilon)$ is of second category.

The continuity of $f_{0}$ and the compactness of $K$ imply that
there are finitely many distinct $x_{1},...,x_{k}\in K$ and positive
$r_{1},...,r_{k}$ such that
\begin{equation}\label{*cover*}
K=\bigcup_{i=1}^{k}B(x_{i},r_{i})
\end{equation}
and for each $i\in\{1,\dots,k\}$ the oscillation of $f_{0}$ on
$B(x_{i},r_{i})$ is less than \begin{equation}\label{*a4*}
\omega=\frac{\varepsilon
(M_1-m_1)}{2(M_2-m_2)}<\frac{\varepsilon}{2}.
\end{equation}

Choose positive $r_{1}',...,r_{k}'$ such that the balls
$B(x_{i},r_{i}')\subseteq B(x_{i},r_{i})$ are disjoint. Using the
weak self-similarity property we can choose for every
$i\in\{1,\dots,k\}$ a set $K_{i} \subseteq B(x_{i},r_{i}')$ and a
bi-Lipschitz map $\phi_{i}\colon K_i \to K$. Let us fix $i\in
\{1,\dots,k\}$. We define the affine function $\psi_{i}\colon
\mathbb{R}\to\mathbb{R}$ such that
\begin{equation} \label{psii} \psi_{i}
\left([m_1,M_1]\right)=[f_0(x_i)-\omega,f_{0}(x_i)+\omega].
 \end{equation}
Suppose $f\in \mathcal{H}_{n}$ and consider $\widehat{f}_{i}\in C(K_{i})$
defined by
$${\widehat{f}_{i}}= \psi_{i}\circ f\circ \phi_{i}.$$
The form of $\psi_i$, \eqref{psii} and \eqref{*a4*} imply
\begin{align*}
\diam {\widehat{f}_{i}}(K_{i})&=\diam \psi_i(f(K)) \leq
\diam \psi_{i}\left([m_2,M_2]\right) \\
&= \frac{M_2-m_2}{M_1-m_1} \diam \psi_{i}\left([m_1,M_1]
\right) \notag \\
&=\frac{M_2-m_2}{M_1-m_1}2\omega=\varepsilon.
\end{align*}
Then $f_0(K_i)\subseteq
[f_0(x_i)-\omega, f_0(x_i)+\omega] \subseteq \widehat{f_i}(K_i)$
and the above
inequality yield for all $x\in K_{i}$
\begin{equation}\label{*a7**}
\left|f_{0}(x)-{\widehat{f}_{i}}(x)\right|\leq \varepsilon.
\end{equation}
Set
$${\widehat{{\mathcal F}}_{i}}=\{ \psi_{i}\circ f\circ \phi_{i}:f\in\mathcal{H}_{n} \}.$$
It follows from \eqref{*a7**} that ${\widehat{{\mathcal
F}}_{i}}\subseteq B\left(f_{0}|_{K_i},\varepsilon \right)$. The maps
$\phi_i\colon K_i \to K$ and $\psi_i \colon \mathbb{R}\to \mathbb{R}$ are
homeomorphisms, hence the map $G_{i}\colon C(K)\rightarrow
C(K_{i})$, $G_{i}(f)=\psi_{i}\circ f\circ \phi_{i}$ is also a
homeomorphism. Since $\mathcal{H}_{n}$ is of second category in $C(K)$, we
obtain that ${\widehat{{\mathcal F}}_{i}}=G_{i}(\mathcal{H}_n)$ is of second
category in $C(K_{i})$.  Set
$$\mathcal{F}_{i}= \left\{ f\in B(f_{0},\varepsilon): f|_{K_{i}}\in
{\widehat{{\mathcal F}}_{i}} \right\}.$$
The map $\widehat{R}_i\colon B(f_0,\varepsilon)\to
B\left(f_0|_{K_i},\varepsilon \right)$, $\widehat{R}_i(f)=f|_{K_i}$
is clearly continuous, and  by Tietze's Extension Theorem it is also
surjective and open. Thus Lemma \ref{catlem} $(ii)$ implies that
$\mathcal{F}_{i}=\widehat{R}_{i}^{-1}\left({\widehat{{\mathcal F}}_{i}}\right)$
is of second category in $B(f_0,\varepsilon)$.
Set
$$\mathcal{F}=\bigcap_{i=1}^{k}\mathcal{F}_{i}.$$
Clearly $\mathcal{F}\subseteq B(f_{0},\varepsilon)$.

\begin{lemma} \label{2ndcat} $\mathcal{F}$ is of second category in $B(f_{0},\varepsilon)$. \end{lemma}

\begin{proof}[Proof of Lemma \ref{2ndcat}] Let $$R\colon B(f_{0},\varepsilon)
\to B\left(f_{0}|_{\bigcup_{i=1}^{k}K_{i}},\varepsilon\right), ~
R(f)=f|_{\bigcup_{i=1}^{k}K_{i}}$$
and for all $i\in \{1,\dots,k\}$
$$R_{i}\colon B\left(f_{0}|_{\bigcup_{i=1}^{k}K_{i}},\varepsilon\right) \to
B\left(f_{0}|_{K_{i}},\varepsilon \right), ~ R_{i}(f)=f|_{K_{i}}.$$
Clearly the map $R$ is continuous, open and surjective. Since
$\mathcal{F}=R^{-1}\left(\bigcap_{i=1}^{k} R_{i}^{-1}\left({\widehat{{\mathcal
F}}_{i}}\right)\right)$, it follows from Lemma \ref{catlem} $(ii)$
that it is enough to prove that $\bigcap_{i=1}^{k}
R_{i}^{-1}\left({\widehat{{\mathcal F}}_{i}}\right)$ is of second
category in
$B\left(f_{0}|_{\bigcup_{i=1}^{k}K_{i}},\varepsilon\right)$. Lemma
\ref{baire} implies that $\mathcal{H}_{n}$ and hence ${\widehat{{\mathcal
F}}_{i}}$ has the Baire property for every $i\in \{1,\dots,k\}$.
Thus there is a non-empty open set $\mathcal{U}_{i}\subseteq C(K_{i})$ such
that ${\widehat{{\mathcal F}}_{i}}$ is co-meager in $\mathcal{U}_{i}$. The sets
$K_{i}$, $i\in\{1,\dots ,k\}$ are disjoint. Hence $\bigcap_{i=1}^{k}
R_{i}^{-1}(\mathcal{U}_{i})\subseteq
B\left(f_{0}|_{\bigcup_{i=1}^{k}K_{i}},\varepsilon\right)$ is a
non-empty open set, and $\bigcap_{i=1}^{k}
R_{i}^{-1}\left({\widehat{{\mathcal F}}_{i}}\right)$ is co-meager in
$\bigcap_{i=1}^{k} R_{i}^{-1}\left(\mathcal{U}_{i}\right)$. Therefore, it is of
second category in
$B\left(f_{0}|_{\bigcup_{i=1}^{k}K_{i}},\varepsilon\right)$.
\end{proof}

Now we return to the proof of Theorem \ref{thinty}. We prove that
$\mathcal{F}\subseteq \mathcal{G}_n$ and then Lemma \ref{2ndcat} will imply that
$\mathcal{G}_n$ is of second category in $B(f_0,\varepsilon)$. Assume that
$g\in \mathcal{F}$. Let $y_0\in \bigcup_{C\in \mathcal{C}} \left(g(C)\setminus
B(\partial g(C),1/n)\right)$ be arbitrary. Then there is a $C_0\in
\mathcal{C}$ such that $B(y_0,1/n)\subseteq  \inter g(C_0)$. The
connectedness of $C_0$ and $g\in B(f_0,\varepsilon)$ yield $y_0\in
f_0(C_0)\subseteq f_{0}(K)$. Hence the definition of $\omega$ and
\eqref{*cover*} imply that there is an $i\in \{1,\dots,k\}$ such
that $y_0 \in [f_{0}(x_{i})-\omega,f_{0}(x_{i})+\omega]$. The
definition of $\mathcal{F}$ yields that there exists an $f\in \mathcal{H}_n$ such
that $g|_{K_{i}}=\psi_{i}\circ f \circ \phi_{i}={\widehat{f}_{i}}$.
Then \eqref{psii} implies $\psi_{i}^{-1}(y_0)\in [m_{1},M_{1}]$, and
$f\in \mathcal{H}_n$ implies
$\dim_{H}f^{-1}\left(\psi_{i}^{-1}(y_0)\right)\geq d_n-1$. By  the
bi-Lipschitz property of $\phi _{i}$ we infer
\begin{align*}
\dim_{H} g^{-1}(y_0)&\geq \dim_{H}{\widehat{f}_{i}}^{-1}(y_0)
=\dim_{H} \phi _{i}^{-1} \left(f^{-1} \left(\psi ^{-1} _{i}
(y_0)\right)\right) \\
&=\dim_{H} f^{-1} \left(\psi^{-1} _{i} (y_0)\right)\geq d_n-1.
\end{align*}
Therefore $g\in \mathcal{G}_n$, and hence $\mathcal{F}\subseteq \mathcal{G}_n$. This completes the proof of Theorem \ref{thinty}.
\end{proof}

It is natural to ask what we can say about the level sets of
\emph{every} $f\in C(K)$. Clearly we cannot hope that for every
$y\in \inter f(K)$ the level set $f^{-1}(y)$ is of small Hausdorff
dimension, since $f$ can be constant on a large set. The opposite
direction is less trivial, it is easy to prove that for every $f\in
C\left([0,1]^2\right)$ for every $y\in \inter f\left([0,1]^2\right)$
we have $\dim_{H}f^{-1}(y)\geq 1=\dim_{tH} [0,1]^{2}-1$. This is not
true in general even for connected self-similar metric spaces. We
have the following counterexample.

\begin{example} \label{example} Let $K=C\times [0,1]$, where $C$ is the von Koch curve. Clearly, $K$ is a
connected self-similar metric space. Let $h\colon C \to [0,1]$ be a homeomorphism and let $f\colon K\to [0,1]$,
$f(x,y)=h(x)$. Theorem \ref{prod} implies that $\dim_{tH} K=\dim_{H} C+1=\frac{\log 4}{\log 3}+1$. For all $y\in [0,1]$ we have $\dim_{H}
f^{-1}(y)=1<\frac{\log 4}{\log 3}=\dim_{tH} K-1$.
\end{example}

\section{Level sets of maximal dimension}

Let $K$ be a compact metric space. If $\dim_{t}K=0$ then the generic
$f\in C(K)$ is one-to-one by Lemma \ref{one-to-one}, thus every non-empty level set is a single
point.

Assume $\dim_t K>0$. Corollary \ref{sup} states that for the
generic $f\in C(K)$ we have $\sup_{y\in \mathbb{R}}
\dim_{H}f^{-1}(y)=\dim_{tH}K-1$. First we prove that in this
statement the supremum is attained.

\begin{theorem}  \label{maxthm} Let $K$ be a compact metric space with $\dim_{t}K>0$.
Then for the generic $f\in C(K)$
$$\max_{y\in \mathbb{R}}\dim_{H}f^{-1}(y)=\dim_{tH} K-1.$$
 \end{theorem}

\begin{proof}

By Theorem \ref{ft} it is sufficient to prove  that for the generic
$f\in C(K)$ there exists a level set of Hausdorff dimension at least
$\dim_{tH}K-1$. Let us fix $x_0\in \supp K$. We will show that for
the generic $f\in C(K)$ we have $\dim_{H} f^{-1}(f(x_0))\geq
\dim_{tH}K-1$. The following lemma is the heart of the proof.

\begin{lemma} \label{lmax2} Let $K_1\subseteq K$ be compact metric spaces with $x_0\in K\setminus K_1$.
Let $d\in \mathbb{R}$ be such that $\dim_{tH}B(x,r)>d$ for all $x\in K_1$
and $r>0$. Then for the generic $f\in C(K)$ either
$\dim_{H}f^{-1}(f(x_0))\geq d-1$ or $f(x_0)\notin f(K_1)$.
\end{lemma}

\begin{proof}[Proof of Lemma \ref{lmax2}] If $d\leq 0$ then the
statement is vacuous, so we may assume $d>0$. We must prove that the
set
$$\mathcal{F}=\left\{f\in C(K): \dim_{H}f^{-1}(f(x_0))\geq d-1 \textrm{ or } f(x_0)\notin f(K_1) \right\}$$
is co-meager in $C(K)$. Let $K_2=B(K_1,\varepsilon_0)$ with such a
small $\varepsilon_0>0$ that $x_0\notin K_2$. Consider
$$\Gamma=\left\{(f,y)\in C(K_2)\times \mathbb{R}: \dim_{H}f^{-1}(y)\geq d-1
\textrm{ or } y\notin f(K_1)\right\}.$$
First assume that $\Gamma$ is co-meager in $C(K_2)\times \mathbb{R}$. Then
we prove that $\mathcal{F}\subseteq C(K)$ is also co-meager. Let $R\colon
C(K)\to C(K_2)\times \mathbb{R}$, $R(f)=(f|_{K_2},f(x_0))$. Clearly $R$ is
continuous, and Tietze's Extension Theorem implies that $R$ is
surjective and open. Thus Lemma \ref{catlem} implies that
$\mathcal{F}=R^{-1}(\Gamma)$ is co-meager.

Finally, we prove that $\Gamma$ is co-meager in $C(K_2)\times \mathbb{R}$.
Lemma \ref{borel} easily implies that $\Gamma$ is Borel, thus has
the Baire property. Hence it is enough to prove by the
Kuratowski-Ulam Theorem \cite[8.41. Thm.]{Ke} that for the generic
$f\in C(K_2)$ for the generic $y\in \mathbb{R}$ we have $(f,y)\in \Gamma$.
Let $\{z_n\}_{n\in \mathbb{N}^{+}}$ be a dense set in $K_1$ and for $i,j\in
\mathbb{N}^{+}$ let us define $B_{i,j}=B(z_i,1/j)$ if $1/j\leq
\varepsilon_{0}$, and $B_{i,j}=K_2$ otherwise. Then for all $i,j\in
\mathbb{N}^{+}$ we have $B_{i,j}\subseteq K_2$ and the conditions of the
lemma yield $\dim_{tH}B_{i,j}>d$. Let $R_{i,j}\colon C(K_2)\to
C(B_{i,j})$, $R_{i,j}(f)=f|_{B_{i,j}}$ and let
$$\mathcal{G}_{i,j}=\left\{f\in C(B_{i,j}): \exists I \textrm{ interval s.t. } \forall
y\in I ~ \dim_{H}f^{-1}(y)\geq d-1\right\}.$$
Set
$$\mathcal{G}=\bigcap_{i,j\in \mathbb{N}^{+}} R_{i,j}^{-1} (\mathcal{G}_{i,j}).$$
It follows from Theorem \ref{ft} that $\mathcal{G}_{i,j}$ is co-meager in
$C(B_{i,j})$ for every $i,j\in \mathbb{N}^{+}$. Corollary \ref{catlem2}
implies that $R_{i,j}^{-1} (\mathcal{G}_{i,j})$ is co-meager in $C(K_2)$, and
as a countable intersection of co-meager sets $\mathcal{G}$ is also
co-meager in $C(K_2)$. We fix $f\in \mathcal{G}$. It is sufficient to verify
that $\Gamma_{f}=\{y\in \mathbb{R}: (f,y)\in \Gamma\}$ is co-meager. Let
$U\subseteq \mathbb{R}$ be an arbitrary open interval. It is enough to
prove that $\Gamma_{f}\cap U$ contains an interval. If there exists
$y_0\in U$ such that $y_0 \notin f(K_1)$ then there is a $\delta>0$
such that $B(y_0,\delta)\cap f(K_1)=\emptyset$, so
$B(y_0,\delta)\cap U$ is an interval in $\Gamma_{f}\cap U$. Thus we
may assume $U\subseteq f(K_1)$. Then there exist $i_0,j_0\in
\mathbb{N}^{+}$ such that $B_0=B_{i_0,j_0}$ satisfies $f(B_0)\subseteq U$.
The definition of $\mathcal{G}$ implies that there is an interval
$I_{f|_{B_0}}\subseteq U$ such that for all $y\in I_{f|_{B_0}}$ we
have
$$\dim_{H}f^{-1}(y)\geq \dim_{H} (f|_{B_{0}})^{-1}(y)\geq d-1.$$
Hence $I_{f|_{B_0}} \subseteq \Gamma_{f}\cap U$, and this completes the proof.
\end{proof}

Now we return to the proof of Theorem \ref{maxthm}. It follows from Fact \ref{equiv}
that $\dim_{tH}K>0$. Since $\dim_{tH}B(x_0,1/n)=\dim_{tH}K$ for
all $n\in \mathbb{N}^{+}$,  the countable stability of the topological
Hausdorff dimension for closed sets implies the following. For all
$n\in \mathbb{N}^{+}$ there exist $r_n>0$  such that the sets
$C_n=B(x_0,1/n)\setminus U(x_0,r_n)$ satisfy $\dim_{tH} C_n>0$ and
$\dim_{tH} C_n \to \dim_{tH} K$ as $n\to \infty$. For all $n\in
\mathbb{N}^{+}$ we put
$$K_n=\left\{x\in C_n: \forall r>0,~ \dim_{tH}(C_n\cap B(x,r))\geq \dim_{tH}
C_n-1/n \right\}.$$
Clearly, the $K_n$'s are compact. First we prove that for all $n\in
\mathbb{N}^{+}$ we have $\dim_{tH} K_n=\dim_{tH}C_n>0$. The definition of
$K_n$ and the Lindel\"{o}f property of $C_n\setminus K_n$ imply that
there are closed balls $B_i$, $i\in \mathbb{N}$ in $C_n$ such that
$\dim_{tH}B_i\leq \dim_{tH} C_n-1/n$ and $\bigcup_{i\in \mathbb{N}}
B_i=C_n\setminus K_n$. Applying the countable stability of the
topological Hausdorff dimension for the closed sets $\{B_i: i\in
\mathbb{N}\}\cup \{K_n\}$ yields $\dim_{tH} K_n=\dim_{tH}C_n$.

Then Fact \ref{equiv} implies $\dim_{t}K_n>0$, and the $K_n$'s
satisfy the conditions of Lemma \ref{lmax1}. Applying Lemma
\ref{lmax1} for the sequence $\langle K_n \rangle _{n\in \mathbb{N}^{+}}$
and the compact set $K$, and applying Lemma \ref{lmax2} for all
$K_n\subseteq K$ with $d_n=\dim_{tH}C_n-2/n$ simultaneously imply
that for the generic $f\in C(K)$ we have $f(x_0) \in f(K_n)$
for infinitely many $n\in \mathbb{N}^{+}$, and for every $n\in \mathbb{N}^{+}$
either $\dim_{tH}f^{-1}(f(x_0))\geq d_n-1$ or $f(x_0)\notin f(K_n)$.
Hence there is a subsequence $\langle n_i \rangle_{i \in \mathbb{N}}$ (that
depends on $f$) such that $\dim_{tH}f^{-1}(f(x_0))\geq d_{n_i}-1$
for all $i\in \mathbb{N}$, that is
$$\dim_{tH}f^{-1}(f(x_0))\geq \lim_{i\to \infty}
\left(\dim_{tH}C_{n_i}-2/n_{i}-1\right)=\dim_{tH}K-1.$$
This concludes the proof.
\end{proof}

\begin{remark} \label{rr} Note that we proved the following stronger statement.
Let $K$ be a compact metric space with $\dim_{t}K>0$. Set for all
$x\in K$
$$ \mathcal{F}_{x}=\left\{f\in C(K):\dim_{H} f^{-1}(f(x))=\dim_{tH}K-1\right\}.$$
Then $\mathcal{F}_x$ is co-meager in $C(K)$ for every $x\in \supp K$.
\end{remark}

The following example shows that the sets $\mathcal{F}_x$, $x\in \supp K$
depend on $x$ indeed in general.

\begin{example} Let $K$ be a self-similar compact metric space with
$\dim_{t}K>1$. It is well-known and easy to prove that for the generic
$f\in C(K)$ the maximum is attained at a unique point, say $x_f$. By
Theorem \ref{<} for the generic $f\in C(K)$ we have
$\dim_{H}f^{-1}(f(x_f))=0< \dim_{t} K-1\leq \dim_{tH}K-1$, thus
$f\notin \mathcal{F}_{x_f}$. Clearly, $\supp K=K$, therefore $\bigcap_{x\in
\supp K} \mathcal{F}_x=\bigcap_{x\in K} \mathcal{F}_x$ is of first category in $C(K)$.
\end{example}

The following theorem shows that we cannot strengthen Theorem
\ref{maxthm} in general, since there exists a compact metric space $K$ such
that the generic $f\in C(K)$ has a unique level set of maximal Hausdorff
dimension. Moreover, $K$ will be the attractor of an iterated function system,
so it will be `homogeneous' to some extent.

\begin{theorem} \label{ifs} There exists a compact set $K\subseteq \mathbb{R}^{2}$ such
that $K$ is an attractor of an iterated function system and the
generic $f\in C(K)$ has a unique level set of Hausdorff dimension
$\dim_{tH}K-1$.
\end{theorem}

\begin{proof} Let $S$ and $C$ be the Smith-Volterra-Cantor set and
the middle-thirds Cantor set, respectively. Let
\begin{align}
&\psi _{1}\colon C\to C \cap \left[0, 1/3\right], ~
\psi_{1}(x)=x/3, \notag \\
&\psi _{2}\colon C\to C \cap \left[2/3, 1\right],~
\psi_{2}(x)=x/3+2/3         \label{psidef}
\end{align}
be the natural similarities of $C$.

Let us define $\alpha_n<1$ $(n\in \mathbb{N}^{+})$ such that $\alpha_n
\searrow 1/3$ as $n\to \infty$. Let $C_n=C_{\alpha_n}$, $n\in
\mathbb{N}^{+}$ be the middle-$\alpha_{n}$ Cantor sets. Then clearly
$\dim_{H} C_{n}\nearrow \dim_{H} C$ as $n\to \infty$. It is easy to
verify that the natural homeomorphisms $\phi _{n}\colon C\to C_{n}$,
$n\in \mathbb{N}^{+}$ are Lipschitz maps. For $r>0$ we denote by
$C_{n}^{r}$ the set that is similar to $C_{n}$, furthermore
$C_{n}^{r}\subseteq [0,r]$ and $\diam C_{n}^{r}=r$. We define
positive numbers $r_{n}$, $n\in \mathbb{N}^{+}$ such that the following
conditions hold for every $n\in \mathbb{N}^{+}$.
\begin{enumerate}[(i)]
\item  \label{*i} There are Lipschitz maps with Lipschitz constant at most $
 1/2$ which map the $n$th level elementary pieces of $S$ onto
$[0,r_{n}]$.
\item \label{*ii} There are Lipschitz maps with Lipschitz
constant at most $
  1/2$ which map the $n$th level elementary pieces of $C$ onto
$C_{n}^{r_{n}}$.
\item \label{*iii} $\sum _{i=n}^{\infty}
r_i\leq 1/2^{2n+2}$.
\end{enumerate}
The $n$th level elementary pieces of $S$ are isometric. They are
of positive Lebesgue measure, since $S$ is of positive measure. It
is well-known that every measurable set with positive measure can be
mapped onto $[0,1]$ by a Lipschitz map \cite[Lemma 3.10.]{BBE},
hence  $\eqref{*i}$ can be satisfied if $r_n$ is small enough.
Moreover, $\eqref{*ii}$ follows from the Lipschitz property of
$\phi _{n}$ for small enough $r_{n}$, and $\eqref{*iii}$ is straightforward.

\placedrawing[h]{
\begin{picture}(110,52)
\thinlines
\drawshadebox{6.8}{9.0}{11.3}{13.19}{14.8}{16.39}{}{0.6}
\drawshadebox{6.8}{9.0}{11.3}{6.8}{8.39}{10.0}{}{0.6}
\drawshadebox{6.8}{9.0}{11.3}{26.0}{27.6}{29.2}{}{0.6}
\drawshadebox{6.8}{9.0}{11.3}{32.4}{34.0}{35.59}{}{0.6}
\drawshadebox{13.1}{15.35}{17.6}{13.19}{14.8}{16.39}{}{0.6}
\drawshadebox{13.1}{15.35}{17.6}{6.8}{8.39}{10.0}{}{0.6}
\drawshadebox{13.1}{15.35}{17.6}{32.4}{34.0}{35.59}{}{0.6}
\drawshadebox{13.1}{15.35}{17.6}{26.0}{27.6}{29.2}{}{0.6}
\drawshadebox{24.79}{27.04}{29.29}{13.19}{14.8}{16.39}{}{0.6}
\drawshadebox{24.79}{27.04}{29.29}{6.8}{8.39}{10.0}{}{0.6}
\drawshadebox{24.79}{27.04}{29.29}{32.4}{34.0}{35.59}{}{0.6}
\drawshadebox{24.79}{27.04}{29.29}{26.0}{27.6}{29.2}{}{0.6}
\drawshadebox{31.1}{33.34}{35.59}{13.19}{14.8}{16.39}{}{0.6}
\drawshadebox{31.1}{33.34}{35.59}{6.8}{8.39}{10.0}{}{0.6}
\drawshadebox{31.1}{33.34}{35.59}{32.4}{34.0}{35.59}{}{0.6}
\drawshadebox{31.1}{33.34}{35.59}{26.0}{27.6}{29.2}{}{0.6}
\drawshadebox{56.0}{66.8}{77.59}{6.8}{8.0}{9.19}{}{0.6}
\drawshadebox{56.0}{66.8}{77.59}{11.6}{12.8}{14.0}{}{0.6}
\drawshadebox{56.0}{66.8}{77.59}{21.2}{22.39}{23.6}{}{0.6}
\drawshadebox{56.0}{66.8}{77.59}{26.0}{27.2}{28.39}{}{0.6}
\drawshadebox{77.59}{83.0}{88.4}{6.8}{7.4}{8.0}{}{0.6}
\drawshadebox{77.59}{83.0}{88.4}{9.19}{9.8}{10.39}{}{0.6}
\drawshadebox{77.59}{83.0}{88.4}{14.0}{14.6}{15.19}{}{0.6}
\drawshadebox{77.59}{83.0}{88.4}{16.39}{17.0}{17.6}{}{0.6}
\drawshadebox{88.4}{91.09}{93.8}{6.8}{7.09}{7.4}{}{0.6}
\drawshadebox{88.4}{91.09}{93.8}{8.0}{8.3}{8.6}{}{0.6}
\drawshadebox{88.4}{91.09}{93.8}{10.39}{10.69}{11.0}{}{0.6}
\drawshadebox{88.4}{91.09}{93.8}{11.6}{11.89}{12.19}{}{0.6}
\drawthickdot{99.19}{6.8}
\drawcenteredtext{21.2}{21.2}{$K_0$}
\drawcenteredtext{66.8}{17.6}{$K_1$}
\drawcenteredtext{83.0}{12.05}{$K_2$}
\drawcenteredtext{100}{8.5}{$x_{\infty}$}
\drawvector{6.8}{4.0}{44.0}{0}{1}
\drawvector{4.0}{6.8}{102.0}{1}{0}
\drawdot{95.1}{8}
\drawdot{96.4}{8}
\drawdot{95.1}{9.3}
\end{picture}
}
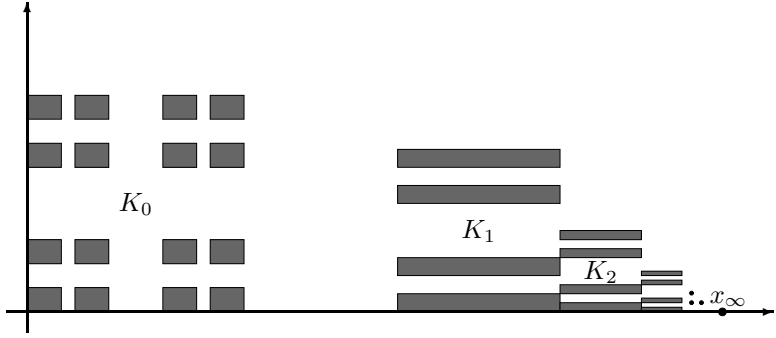
{Illustration to the construction of $K$}{fig:IFS}

Let $K_{0}=S\times C$, $x_{\infty}=\left(2+\sum _{i=1}^{\infty}
r_{i},0\right)$ and for all $n\in \mathbb{N}^{+}$ let
\begin{align*}  I_{n}&=\textstyle{\left[2+\sum _{i=1}^{n-1} r_{i},2+\sum _{i=1}^{n}
r_{i}\right]}, \\
K_{n}&=I_{n}\times C_{n}^{r_{n}}, \\
K&= \bigcup _{n=0}^{\infty} K_{n}\cup \{x_{\infty}\}, \\
\widetilde{K}_{n}&= \bigcup_{i=n}^{\infty}K_i \cup \{x_{\infty}\}, \\
\widehat{K}_{n}&=\bigcup _{i=0}^{n}K_{i}.
\end{align*}
Clearly, all the sets defined above are compact.

First we prove that $K$ is an attractor of an IFS. Recall that the
$\varphi_i$'s and $\psi_j$'s are the natural homeomorphisms of $S$
and $C$, respectively. For the more precise definition see
\eqref{fidef} and \eqref{psidef} again. Let us define for $i,j\in
\{1,2\}$ the maps $\Psi_{i,j}:K\to K_0$ such that
\begin{equation*} \Psi _{i,j}(x)=
\begin{cases} (0,0) & \textrm{ if } x\in K\setminus K_{0}, \\
\left(\varphi _{i}(x),\psi _{j}(x)\right) & \textrm{ if } x\in K_0.
\end{cases}
\end{equation*}
Clearly, the $\Psi _{i,j}$'s are Lipschitz maps with $\Lip(\Psi
_{i,j})\leq 1/2$, and $\bigcup_{i,j\in \{1,2\}}\Psi_{i,j}(K)=K_0$.
For all $n\in \mathbb{N}^{+}$ and $(i,j)\in
\{(1,2),(2,1),(2,2)\}$ let us define the sets $K_{i,j,n}$ to be the top left,
the top right and the
bottom right $n$th level `elementary pieces' of the bottom left
$(n-1)$st `elementary piece' of $K_0$, that is,
$$K_{i,j,n}= \left(\varphi _{1}^{n-1}\circ \varphi_i \right)(S) \times
\left( \psi _{1}^{n-1}\circ \psi_j \right)(C).$$
These are clearly disjoint subsets of $K_0$.
It follows from $\eqref{*i}$ and
$\eqref{*ii}$ that for all $n\in \mathbb{N}^{+}$ and $(i,j)\in
\{(1,2),(2,1),(2,2)\}$ there exist surjective Lipschitz maps
$$\varphi_{i,n}\colon \left(\varphi _{1}^{n-1} \circ \varphi_i \right)(S)\to
I_{n} ~ \textrm{ and } ~ \psi_{j,n}\colon \left(\psi
_{1}^{n-1} \circ \psi_j \right)(C)\to C_{n}^{r_{n}}$$
with Lipschitz constant at most $1/2$. Let $\Psi\colon K\to
K\setminus K_0$ be the following map.
\begin{equation*} \Psi(x)=
\begin{cases} x_{\infty} & \textrm{ if } x\in K\setminus K_{0} \textrm{ or }
x=(0,0), \\
\left(\varphi_{i,n}(x),\psi_{j,n}(x)\right) & \textrm{ if } x\in
K_{i,j,n}.
\end{cases}
\end{equation*}
The $K_{i,j,n}$'s, $K\setminus K_0$ and $\{(0,0)\}$ are disjoint
sets with union $K$, so $\Psi$ is well-defined. Clearly $\Psi$ maps
$K_{i,j,n}$ onto $K_n$, and hence $\Psi (K)=K\setminus K_0$. Thus $K=\bigcup
_{i,j\in \{1,2\}} \Psi _{i,j}(K)\cup \Psi (K)$. Therefore, it is enough
to prove that $\Psi$ is a Lipschitz map with $\Lip(\Psi)\leq 1/2$,
that is for all $x,z\in K$
\begin{equation} \label{psi} |\Psi (x)-\Psi (z)|\leq \frac{|x-z|}{2}.  \end{equation}
If $x,z\in K\setminus K_{0}$ then $\Psi (x)=\Psi (z)=x_{\infty}$,
thus \eqref{psi} follows.
\\ If $x\in K_{0}$ and $z\in K\setminus K_{0}$, then clearly $|x-z|\geq 1$. On the other
hand,
$\eqref{*iii}$ implies
\begin{align*} |\Psi (x)-\Psi (z)|&\leq \diam (K\setminus K_0)\leq
\sqrt{\left(\sum_{i=1}^{\infty} r_i\right)^2+(r_1)^2} \\
&<2\sum_{i=1}^{\infty}r_i \leq 1/8,
\end{align*}
therefore \eqref{psi} follows.
\\ If $x=(x_1,x_2)\in K_{0}$ and $z=(z_1,z_2)\in K_{0}$ then we may
assume that
\begin{equation} \label{maxx} \max \{|z_1|,|z_2|\}\leq \max\{|x_1|,|x_2|\}.
\end{equation}
If $x=(0,0)$ then $z=(0,0)$ and  we are done. We may assume $x\in K_{i,j,n}$,
where $n\in \mathbb{N}^{+}$ and $(i,j)\in \{(1,2),(2,1),(2,2)\}$. If $z\in
K_{i,j,n}$ then \eqref{psi} follows, since $\Psi$ is Lipschitz on
$K_{i,j,n}$ with Lipschitz constant at most $1/2$. Hence we may
assume $z\in K\setminus K_{i,j,n}$. Then \eqref{maxx} implies
$\Psi(x),\Psi(z)\in \widetilde {K}_{n}$. By the definition of $
\widetilde {K}_{n}$ and $\eqref{*iii}$
$$|\Psi (x)-\Psi (z)|\leq \diam
\widetilde{K}_{n}<2\sum _{i=n}^{\infty} r_{i}\leq 1/2^{2n+1}.$$
The minimum distance between distinct $n$th level elementary
pieces of $S$ and $C$ is $1/2^{2n}$ and $1/3^{n}$, respectively.
Since $K_0=S\times C$,
\begin{align*} |x-z|&\geq \dist \left(K_{i,j,n},K\setminus K_{i,j,n}
\right) \\
&\geq \dist \left(K_{i,j,n},K_{0}\setminus K_{i,j,n} \right) \geq
1/2^{2n}.
\end{align*}
These imply \eqref{psi}, and hence $K$ is an attractor of an IFS.

Finally, we prove that the generic $f\in C(K)$ has a unique level
set of Hausdorff dimension $\dim_{tH} K-1$.

By Theorem \ref{maxthm}  the generic $f\in C(K)$ has at least one
level set of Hausdorff dimension $\dim_{tH} K-1$. Hence it is enough
to show that for the generic $f\in C(K)$ for all $y\neq
f(x_{\infty})$ we have $\dim_{H}f^{-1}(y)<\dim_{tH}K-1$. From Fact
\ref{equiv} follows $\dim_{tH} K_{0}=0$, clearly $\dim_{tH}
\{x_{\infty}\}=0$ and Theorem \ref{prod} implies $\dim_{tH} K_{n}
-1=\dim_{H} C_{n}$. This, together with the countable stability of
the topological Hausdorff dimension for closed sets and the
definition of $C_n$ yield
$$\dim_{tH} K-1=\sup _{n\in \mathbb{N}^{+}} \dim_{tH} K_{n}-1=
\sup_{n\in \mathbb{N}^{+}}\dim_{H} C_{n}=\dim_{H} C.$$
 Assume to the contrary that there exists
$\mathcal{F}\subseteq C(K)$ such that $\mathcal{F}$ is of second category and for
every $f\in \mathcal{F}$ there exists $y_{f}\neq f(x_{\infty})$ such that
$\dim_{H}f^{-1}(y_{f})=\dim_{H} C$. Then $f^{-1}(y_{f})\subseteq
K\setminus \{x_{\infty}\}$, and by the compactness of $f^{-1}(y_f)$
there exists an $n_{f}\in \mathbb{N}^{+}$ such that $f^{-1}(y_{f})\subseteq
\widehat{K}_{n_f}$. Set
$$\mathcal{F}_{n}=\left\{f\in \mathcal{F}: f^{-1}(y_{f})\subseteq
\widehat{K}_{n}\right\}.$$
Since $\mathcal{F}=\bigcup _{n=1}^{\infty} \mathcal{F}_{n}$, there exists $n_{0}\in \mathbb{N}^{+}$ such that $\mathcal{F}_{n_{0}}$
is of second category in $C(K)$. We obtain from Corollary
\ref{catlem2} $(i)$ that
$$\widehat{\mathcal{F}}_{n_0}=\left\{f|_{\widehat{K}_{n_0}}: f\in
\mathcal{F}_{n_{0}}\right\}$$
is of second category in $C\left(\widehat{K}_{n_0}\right)$. The
definition of $\widehat{\mathcal{F}}_{n_0}$ implies that for every $f\in
\widehat{\mathcal{F}}_{n_0}$ we have $\dim_{H} f^{-1}(y_{f})=\dim_{H} C$.
By Theorem \ref{ft}  for the generic $f\in
C\left(\widehat{K}_{n_0}\right)$ every level set is of Hausdorff
dimension at most
$$\dim_{tH}
\widehat{K}_{n_0}-1=\sup_{1\leq n\leq n_0}\dim_{tH} K_n
-1=\dim_{H}C_{n_0}<\dim_{H}C,$$
a contradiction. This concludes the theorem.
\end{proof}

\begin{question} Does there exist an attractor of an \emph{injective} iterated function system $K$
such that the generic $f\in C(K)$ has a unique level set of
Hausdorff dimension $\dim_{tH}K-1$? \end{question}

\section{The dimension of the graph of the generic continuous function}

The graph of the generic $f\in C([0,1])$ is of Hausdorff dimension
one, this is a result of R. D. Mauldin and S. C. Williams \cite[Thm. 2.]{MW}. We
generalize the cited theorem for arbitrary compact metric spaces.
Let $K$ be a compact metric space, then for the generic $f\in C(K)$
the graph of $f$ is of Hausdorff dimension $\dim_{H}K$. We prove an
analogous theorem for the topological Hausdorff dimension, for the
generic $f\in C(K)$ the graph of $f$ is of topological Hausdorff
dimension $\dim_{tH}K$.

\begin{definition} If $f\in C(K)$ let us define
$$\widetilde{f}\colon K\to \graph(f), \quad \widetilde{f}(x)=(x,f(x)).$$
\end{definition}

Clearly $\widetilde{f}$ is continuous and one-to-one, so it is a
homeomorphism between $K$ and $\graph(f)$.

\begin{theorem} \label{graph0} If $K$ is a compact metric space then
for the generic $f\in C(K)$
$$\dim_{H} \graph(f)=\dim_{H}
K.$$  \end{theorem}

Theorem \ref{graph0} follows from the following more general theorem
applied with $C=K$. We  need this slight generalization in order to
prove Theorem \ref{graph2}.

\begin{theorem} \label{graph1} If $C \subseteq K$ are compact metric spaces then
for the generic $f\in C(K)$
$$\dim_{H} \graph(f|_C) = \dim_{H} C.$$
\end{theorem}

\begin{proof}[Proof of Theorem \ref{graph1}]
First note that $\graph(f|_C) = \widetilde{f} (C)$.
For every $f\in C(K)$ the map $\widetilde{f}^{-1}$ is a
projection from $\widetilde{f}(C)$ onto $C$. Since  the Hausdorff
dimension cannot increase under a Lipschitz map,
$\dim_{H}\widetilde{f}(C)\geq \dim_{H} C$. For the opposite
direction it is enough to prove that
$$\mathcal{F}= \left\{f\in C(K): \dim_{H} \widetilde{f}(C)\leq \dim_{H} C\right\}$$
is a dense $G_{\delta}$ set in $C(K)$. We may assume $\dim_{H}
C<\infty$. First we show that $\mathcal{F}$ is a $G_{\delta}$ set. Let us
define for all $n\in \mathbb{N}^{+}$
$$\mathcal{F}_{n}=\left\{f\in C(K):\mathcal{H}^{\dim_{H}
C+1/n}_{1/n}\Big(\widetilde{f}(C)\Big)<1/n\right\}.$$
As $\widetilde{f}(C)$ is compact, it is straightforward that the $\mathcal{F}_{n}$'s are open and clearly
$\mathcal{F}=\bigcap_{n\in \mathbb{N}^{+}}\mathcal{F}_{n}$. Thus $\mathcal{F}$ is a $G_{\delta}$ set.

Finally, we show that $\mathcal{F}$ is dense in $C(K)$. If $f\in C(K)$ is
Lipschitz, then clearly $\widetilde{f}$ is Lipschitz with
$\Lip(\widetilde{f} )\leq \Lip(f)+1$, and hence $\dim_{H}
\widetilde{f}(C) \leq \dim_{H} C$. Therefore, it is enough to prove
that $\mathcal{G}=\{f\in C(K): f \textrm{ is Lipschitz}\}$ is dense in
$C(K)$. This fact is well-known but one can also see it directly,
since it is easy to show that $cf\in \mathcal{G}$, $f+g\in \mathcal{G}$ and $fg\in
\mathcal{G}$
 for all $f,g\in \mathcal{G}$ and $c\in \mathbb{R}$. Therefore,
$\mathcal{G}$ forms a subalgebra in $C(K)$. Finally, we may assume $\#K\geq
2$, and the Lipschitz functions $\{\varphi_{x_0}\}_{x_0\in K}$,
$\varphi_{x_0}\colon K \to \mathbb{R},$ $\varphi_{x_0}(x)=d_{K}(x_0,x)$
show that $\mathcal{G}$ separates points of $K$ and $\mathcal{G}$ vanishes at no
point of $K$. Hence the Stone-Weierstrass Theorem \cite[12.9.]{CA}
implies that $\mathcal{G}$ is dense. This completes the proof. \end{proof}

\begin{theorem}  \label{graph2} If $K$ is a compact metric space then
for the generic $f\in C(K)$
$$\dim_{tH} \graph (f)=\dim_{tH} K.$$  \end{theorem}

\begin{proof} For every $f\in C(K)$ the map $\widetilde{f}^{-1}$ is an injective projection from
$\graph(f)$ onto $K$, hence it is a Lipschitz homeomorphism. Thus
Theorem \ref{prop} implies that $\dim_{tH} \graph(f)\geq \dim_{tH}
K$. For the opposite direction choose a basis $\mathcal{U}$ of $K$ such that
$\dim_{H} \partial U\leq \dim_{tH} K-1$ for all $U\in \mathcal{U}$, we can
do this by Theorem \ref{min}. We may assume that $\mathcal{U}$ is countable.
Suppose $U\in \mathcal{U}$ is arbitrary. By applying Theorem \ref{graph1}
for $C=\partial U$ we infer that there exists a co-meager set
$\mathcal{F}_{U}\subseteq C(K)$ such that for all $f\in \mathcal{F}_{U}$ we have
$\dim_{H} \widetilde{f}(\partial U)=\dim_{H} (\partial U)\leq
\dim_{tH} K-1$. The basis $\mathcal{U}$ is countable, and hence
 $\mathcal{F}=\bigcap_{U\in
\mathcal{U}} \mathcal{F}_{U}$ is co-meager in $C(K)$. Assume $f\in \mathcal{F}$, it is
enough to prove that $\dim_{tH} \graph(f)\leq \dim_{tH} K$. Since
$\widetilde{f}$ is homeomorphism we obtain that
$\mathcal{V}=\left\{\widetilde{f}(U): U\in \mathcal{U}\right\}$ is a basis of
$\graph(f)$ and $\partial \widetilde{f}(U)=\widetilde{f}(\partial
U)$ for all $U\in \mathcal{U}$. That is,
$$\dim_{H}\partial
V=\dim_{H}\partial \widetilde{f}(U)=\dim_{H} \widetilde{f}(\partial
U)=\dim_{H}\partial U\leq \dim_{tH} K-1$$
for all $V=\widetilde{f}(U)\in \mathcal{V}$.
Thus $\dim_{tH} \graph(f)\leq \dim_{tH} K$, and this completes the proof.
\end{proof}

\noindent \textbf{Acknowledgment.} We are indebted to an anonymous
referee for his valuable comments and for suggesting Example \ref{example}.

\end{document}